\documentclass{article}

\usepackage{arxiv}

\usepackage{amsmath,amssymb,amsthm}
\usepackage{cite,color}
\usepackage{bm}
\usepackage{tikz}
\usetikzlibrary{arrows.meta}
\usepackage{orcidlink}

\newtheorem{theorem}{Theorem}
\newtheorem{lemma}[theorem]{Lemma}
\newtheorem{proposition}[theorem]{Proposition}
\newtheorem{corollary}[theorem]{Corollary}

\newtheorem{remark}[theorem]{Remark}
\newtheorem{problem}[theorem]{Problem}
\newtheorem{example}[theorem]{Example}

\DeclareMathOperator{\diag}{diag}
\DeclareMathOperator{\tr}{tr}

\DeclareMathOperator{\Skew}{Skew}

\newcommand{\blambda}{\boldsymbol{\lambda}}

\begin{document}
\title{Isospectral Steering}
\author{Ralph Sabbagh\orcidlink{https://orcid.org/0000-0002-2020-5420}, {\em Member, IEEE},  and Tryphon T.\ Georgiou\orcidlink{https://orcid.org/0000-0003-0012-5447}, {\em Life Fellow, IEEE}%
\thanks{The authors are with the Department of Mechanical and Aerospace Engineering, University of California, Irvine, Irvine, CA, USA, emails: rsabbag1@uci.edu, tryphon@uci.edu.}
\thanks{Supported by the National Science Foundation under grant ECCS-2347357, the Air Force of Scientific Research under FA9550-24-1-0278, and the Army Research Office under W911NF-22-1-0292.}
}

\date{}

\maketitle

\begin{abstract}
We study the controllability of the differential Lyapunov equation
under isospectral rotation of a linear gradient field. Specifically, control is
effected by a symmetric time-varying gain-matrix  
constrained to have fixed eigenvalues; that is, by exclusively modulating the eigen-vectors of the state matrix 
and not its eigenvalues.  Motivation for this problem stems from a certain type of control objectives (minimum shear/attention) aimed to reduce anisotropic deformation
when ensembles are steered by a common law--optimality necessitates constancy of eigenvalues. In the paper we introduce and motivate this type of isospectral steering, and describe the reachable set of covariances for any specified terminal time and eigenvalues of the gain matrix. The theory we develop is intimately linked to multilinear algebra as well as to positive linear algebra and the Birkoff-von Neumann theorem for doubly stochastic matrices.
\end{abstract}


\section{Introduction}\label{sec:intro}

We study the controllability of the differential Lyapunov equation and of the corresponding state-transition dynamics with isospectral control laws.
In this, it is the time-varying control gain that remains isospectral, and not the state covariances nor the state-transition matrices.
Our interest in this niche of the covariance control literature stems from the recent work \cite{Sabbagh26} where it was shown that for a class of performance objectives, a hidden Lax lift of the Hamilton-Jacobi dynamics, descends onto the optimal gain, enforcing that it remains isospectral.

Specifically, we consider an ensemble of dynamical systems or particles situated at locations $\mathbf x_t^{(k)}\in \mathbb R^n$, for $k\in\{1,\ldots,N\}$, that follow the gradient of a time-varying quadratic potential $\phi(t,\mathbf x)=\frac12 \mathbf x^\top A_t \mathbf x$; the scheduling of $\{A_t\mid t\in[0,T ]\}$ constitutes the control law. Our task is to steer the (sample) covariance
\[
\Sigma_t=\frac{1}{N}\sum_{k=1}^N \mathbf x_t^{(k)}{\mathbf x_t^{(k)}}^\top
\]
 of a zero-mean point cloud that obeys the Lyapunov equation
\begin{equation}\label{eq:lyap}
\dot \Sigma_t=A_t\Sigma_t+\Sigma_tA_t
\end{equation}
between specified end-point values $\Sigma_0$, $\Sigma_{T}\in \mathbb S^n_{++}$, the space of $n\times n$, symmetric and positive definite matrices. 
It was pointed out in \cite{Sabbagh26} that for certain control objectives that penalize the spread of the eigenvalues of $A_t$, the Euler–Lagrange/Pontryagin optimality conditions dictate that $A_t$ remains isospectral.
Examples of such objectives include
the time integral of the difference between the maximal and minimal eigenvalues of $A_t$ \cite{Sabbagh26}, and the time integral of $\tr(A_t^2)$ that quantifies {\em shear} deformation, equivalently {\em Brockett's attention} \cite{Brockett1997}, \cite{sabbagh2025minimizing}, that we further elaborate below and analyze in the body of the paper.

The rationale behind the concept of Brockett's attention \cite{Brockett1997}, i.e., keeping the gradient $\nabla (A\mathbf x)=A^\top$ small, stems from the fact that the gradient of the vector field quantifies the sensitivity of the control action $A\mathbf x$ on the system's location and, thereby, intuitively, the ``attention'' that the controller has to pay on where exactly the system resides so as to generate the required action. On the flip side, $A_t$ is the traceless symmetric part of the strain tensor (which, here, is symmetric already), and thereby quantifies the rate of shape distortion. Specifically, it measures how a flow causes nearby material elements to be stretched in some directions and compressed in others, without the isotropic part of uniform expansion or contraction (that is encoded in the trace). In other words, large shear means that neighboring systems separate at substantial rates depending on directions.Both of these interpretations are relevant in modern-day applications where an ensemble of, e.g., satellite constellations or multi-agent drone swarms, are guided as a collective to reposition into new formations~\cite{Oh2015survey,Beard2001coordination,Scharf2004survey}.

Other paradigms where control of (possibly non-isotropic) shear may be desirable include soft robots, artificial muscles, and dielectric elastomers~\cite{Rus2015soft,Pelrine2000high} that act as continuous spatial ensembles rather than discrete rigid joints, driven by applied fields (pneumatic, electrical, or magnetic) to achieve a desired shape so as to, for instance, squeeze through a narrow passage. Also in biomedicine, researchers use acoustic fields to steer ensembles of living cells and nanoparticles into specific distributions~\cite{Ozcelik2018acoustic,Ding2012onchip}. In all such applications, fluctuating shear forces may have an adverse effect, ripping biological cells apart~\cite{Leverett1972red} and inducing mechanical fatigue in soft materials~\cite{Mars2004factors}. Thus, in general, minimizing shear in distributed control may have several enabling consequences.

After recognizing the control relevance of minimizing shear, it is important to note that already in the context of steering between specified ellipsoidal formations the optimization problem is not convex.  However, as we will see, the solution to this optimization problem possesses the rather remarkable property that the control remains isospectral along optimal paths. This property suggests turning the question around, and instead of numerically seeking a control protocol $\{A_t\mid t\in[0,T]\}$ that may not even be unique, seek to first characterize the set of  covariances $\Sigma_{T}$ that are reachable from a starting $\Sigma_0$ by employing an isospectral control protocol with specified eigenvalues and over a specified time interval. 

Thus, in the present paper we address this precise question, of what states can be reached using a fixed ``spectral signature''? More precisely, we study the reachable set of \eqref{eq:lyap} over the time interval $[0,T]$ when the time-varying generator $A_t$ is a symmetric matrix constrained to having a prescribed spectrum. Besides minimizing attention/shear, isospectrality of the control law arises with other cost functionals that penalize the spread of eigenvalues of the control gain. Moreover, studying isospectral steering appears relevant in quantum control in cases where a control Hamiltonian has prescribed fixed energy levels and the control consists of rotating the effective couplings in a chosen basis, e.g., see \cite[Example, page 98]{zanardi1999holonomic}, \cite{karimipour2005exact}.

The present paper is organized as follows.
Section~\ref{sec:problem} defines the isospectral steering problem.
Section~\ref{sec:shear} highlights the origin of the isospectrality of the control law in the context of minimizing shear, and shows that this property persists for a broad class of cost functionals.
Section~\ref{isospcon} presents necessary and sufficient conditions for isospectral controllability.
Section~\ref{sec:background} develops the requisite background on exterior powers and compound matrices, and collects the technical lemmas needed for the proof of the main result.
Section~\ref{sec:proof-main} gives the proof of Theorem~\ref{maintheorem}.
Section~\ref{sec:time-reversal} exploits time-reversal symmetry to derive additional necessary and sufficient conditions, and establishes finite-time controllability for any nontrivial traceless spectrum.
Section~\ref{sec:numerical} illustrates the constructive nature of the controllability proof with a three-dimensional numerical example.
Section~\ref{sec:conclusion} provides concluding remarks and highlights open questions.

\subsection*{Notation}

As usual, $GL(n,\mathbb R)$ denotes 
the group of $n\times n$ invertible matrices and 
$GL_+(n,\mathbb R)$ those with positive
determinant. Then, 
 $O(n)$ denotes the group of $n\times n$ orthogonal matrices, 
$SO(n)\subset O(n)$ those with determinant $+1$, and $SL(n,\mathbb R)\subset GL_+(n,\mathbb R)$ again those of determinant $1$.
Then, $\mathbb S^n$ denotes the $n\times n$ real symmetric matrices while $\mathbb S^n_0$ denotes those with trace $0$, $\mathbb S^n_{++}$ denotes the positive definite ones, and $\Skew(n)$ denotes the skew-symmetric matrices of size $n\times n$.
For $M\in \mathbb R^{k\times \ell}$,
$\sigma_1(M)\ge\cdots\ge\sigma_n(M)>0$ denote the singular values in decreasing order,
$\|M\|=\sigma_1(M)$ the spectral norm, and 
$\|M\|_F = (\tr M^\top M)^{1/2}$ the Frobenius norm, with $\tr$ being the trace. The commutator of two matrices $A,B$ is denoted by
$[A,B]:=AB-BA$.
Finally,
$\mathbf e_1,\ldots,\mathbf e_n$ the standard basis of $\mathbb R^n$.

\section{The Problem of isospectral controllability}\label{sec:problem}
The main theme of this paper is to study the controllability of the differential Lyapunov equation \eqref{eq:lyap}, from an initial condition $\Sigma_0>0$ by way of a time-varying control gain-matrix that is constrained to be isospectral, i.e., via
\begin{align}\label{eq:control}
    A_t = U_tDU_t^\top,~~U_t\in O(n),
\end{align}
where $t\mapsto U_t$ is measurable and $D$ diagonal, with fixed eigenvalues
$\lambda_1(D)\geq\lambda_2(D)\geq\cdots \geq\lambda_n(D)$.
Evidently, the spectrum of $A_t\in\mathbb S^n$ coincides with that of $D$, and remains constant and prespecified for $t\in[0,T]$. Thus, we address the following.

\begin{problem}\label{prob1}
    Determine the space of achievable covariances at the end of the time interval $[0,T]$, i.e., the reachable set 
\begin{align*}
    \mathcal{R}({\Sigma_0,D,T }):=\{\Sigma_{T} \in \mathbb S^n_{++}\mid \Sigma_0, \mbox{ and where (\ref{eq:lyap}-\ref{eq:control}) hold}\},
\end{align*}
over choices of the control parameter $U_t\in O(n)$.
\end{problem}

\begin{remark}
    For the most part, for simplicity and without loss of generality, we assume that 
\(\det(\Sigma_0)=\det(\Sigma_{T})\)
and that $A_t$ is traceless, i.e., in $\mathbb S^n_0$. In fact, a non-zero trace of $A_t$ induces isotropic scaling of the reachable set in that,
 \begin{align}\nonumber
       \mathcal{R}(\Sigma_0,{D,T })&=e^{-2T \alpha}\mathcal{R}(\Sigma_0,{(D+\alpha I_n),T }),~~\alpha\in\mathbb{R}.
 \end{align}
\end{remark}
\vspace*{.05in}

\begin{remark}
    The choice
    \[
    A_t = A_{\rm OT}:= \frac{1}{T} \log( \Sigma_0^{-1/2}(\Sigma_0^{1/2}\Sigma_1\Sigma_0^{1/2})^{1/2}\Sigma_0^{-1/2})
    \]
    has the desired effect to steer $\Sigma_0$ to $\Sigma_T$, since
    \[
 \Sigma_T=e^{A_{\rm OT}T}\Sigma_0e^{A_{\rm OT}T}.
    \]
    This is the unique symmetric matrix that does so and is constant in time. Thus, if the prescribed spectrum differs from the spectrum of $A_{\rm OT}$, isospectral steering may still be possible, but in this case the control protocol will have to be time-varying; the eigenvectors will need to rotate accordingly.
\end{remark}

\section{Motivation: Isospectrality of shear-minimizing control}\label{sec:shear}

We begin by considering the control law \(\{A_t\}_{t\in[0,t_{\mathrm{fin}}]}\) that minimizes shear/Brockett's attention (equation \eqref{eq:integralcost} below). The development here is provided as a way of example to highlight that a uniform-in-time penalty on the spread of the eigenvalues of $A_t$ often leads to dictating that $A_t$ remains isospectral. This observation was first pointed out in \cite{Sabbagh26}, where the authors dealt with a functional that quantifies the spread
\[
\lambda_{\mathrm{max}}(A_t)-\lambda_{\mathrm{min}}(A_t)
\]
of the eigenvalues of $A_t$,
deducing similarly that $A_t$ remains isospectral. As we will explain below, there is a large class of functionals that share similar properties. Thus, we begin by considering the following problem.
\begin{problem}\label{prob2}Determine \(\{A_t\}_{t\in[0,T]}\)
    with $A_t\in\mathbb S_0^n$, i.e., symmetric and traceless, to minimize
\begin{equation}\label{eq:integralcost}
J(A_\cdot)=\int_0^{T}\tr(A_t^2)\,dt
\end{equation}
subject to \eqref{eq:lyap}, and the terminal conditions
$
\Sigma_0,\Sigma_{T}\in\mathbb S_{++}^n$,
with \(\det(\Sigma_0)=\det(\Sigma_T)\).
\end{problem}

We now derive the stationary conditions on \((\Sigma_t,A_t)\) and show that these require \(A_t\) to remain isospectral on \([0,T]\).
With \(\Lambda_t\in\mathbb S^n\) a Lagrange multiplier (co-state), the augmented functional is
\[
\mathcal L
=
\int_0^{T}
\Big(
\tr(A_t^2)
+
\tr\big(\Lambda_t(\dot\Sigma_t-A_t\Sigma_t-\Sigma_tA_t)\big)
\Big)\,dt.
\]
Setting the variation with respect $\Lambda_t$, $\Sigma_t$, and $A_t$ equal to zero we obtain the Euler--Lagrange/Pontryagin conditions
\begin{subequations}
   \begin{equation}\label{eq:first}
       \dot\Sigma_t=A_t\Sigma_t+\Sigma_tA_t,
\qquad
\dot\Lambda_t=-\Lambda_tA_t-A_t\Lambda_t,
   \end{equation}
and
\begin{equation}\label{eq:second}
2A_t-(\Sigma_t\Lambda_t+\Lambda_t\Sigma_t)
+\frac{2}{n}\tr(\Lambda_t\Sigma_t)I=0.
\end{equation}
\end{subequations}
Define $L_t:=\Lambda_t\Sigma_t$, and note that the symmetric part of $L_t$ is
\begin{equation}\label{eq:sym}
\frac12(L_t+L_t^\top)=A_t+\frac1n\tr(L_t)I,
\end{equation}
and that, from \eqref{eq:first},
\begin{align*}
\dot L_t
&=
\dot\Lambda_t\Sigma_t+\Lambda_t\dot\Sigma_t
=
L_tA_t-A_tL_t\\
&=
[L_t,A_t]=[\Omega_t,A_t],
\end{align*}
where $\Omega_t:=\frac12(L_t-L_t^\top)$ is the skew-symmetric part (since the symmetric part of $L_t$ in \eqref{eq:sym} already commutes with $A_t$).
Since $\dot L_t=[L_t,A_t]$, $L_t$ remains isospectral\footnote{To see this, consider the fundamental matrix $\Phi_t$ satisfying $\dot \Phi_t=\Phi_tA_t$, with $\Phi_0=I$. It is immediate that $L_t = \Phi_t^{-1}L_0\Phi_t$, and hence, $L_t$ remains similar to $L_0$ throughout $[0,T]$.},
and hence, $\tr(L_t)$ remains constant throughout $[0,T]$. 
Moreover,
$[\Omega_t,A_t]$ is symmetric, because $A_t$ is symmetric and $\Omega_t$ antisymmetric, and therefore $\dot L_t$ is symmetric and the skew-symmetric part  \(\Omega_t\equiv\Omega\) remains constant. So finally,
\[
\dot L_t=\dot A_t+\underbrace{\frac1n\tr(\dot L_t)I}_{=0}=[\Omega,A_t],
\]
from which we deduce that the (traceless) $A_t$ satisfies
\begin{equation}\label{eq:Aiso}
\dot A_t=[\Omega,A_t].
\end{equation}
Condition \eqref{eq:Aiso} readily establishes the following proposition.

\begin{proposition}
    The solution to Problem \ref{prob2} is isospectral, namely, the eigenvalues of $A_t$ remain constant for $t\in[0,T]$.
\end{proposition}

The isospectral property persists for other cost functionals, e.g., functionals of the form
\begin{equation}\label{eq:ff}
J_f(A_\cdot):=\int_0^{T} \tr(f(A_t))dt,
\end{equation}
if these were to replace the quadratic expression in \eqref{eq:integralcost}. To see this,
note that in the above analysis, \eqref{eq:second} is then replaced by\footnote{Here, we assume that $f$ is a scalar function, $f^\prime$ denotes the derivative, and that $f,f^\prime$ are defined on the spectrum of $A_t$, with functions of  $A_t$ computed by spectral calculus. One may however consider more general scalar functions of $A_t$ that are not of the form $\tr(f(A_t))$, with often similar results.}
\begin{equation}\label{eq:secondprime}
f^\prime(A_t)-(\Sigma_t\Lambda_t+\Lambda_t\Sigma_t)
+\frac{2}{n}\tr(\Lambda_t\Sigma_t)I=0. \tag{\ref{eq:second}'}
\end{equation}
Then, the symmetric part of $L_t$ is
\[
\frac12(L_t+L_t^\top)=f^\prime(A_t)+\frac{1}{n}\tr(L_t)I,
\]
and in precisely the same manner, $L_t$ remains isospectral. It follows that
\[
\frac{d}{dt}f^\prime(A_t)=[\Omega,A_t]
\]
as before, and if we define $M_t:=f^\prime(A_t)$, then $M_t$ commutes with $A_t$; it easily follows that
\[
\frac{d}{dt}\tr(M_t^k)=0, \mbox{ for }k\in\{1,2,\ldots\}.
\]
Thus, $M_t=f^\prime(A_t)$ is isospectral, and provided $f^\prime(\cdot)$ is injective on the specrum of $A_t$, it follows that $A_t$ remains isospectral as well. We summarize as follows.

\begin{proposition}\label{prop:6}
 The solution to the variant of Problem \ref{prob2}, where \eqref{eq:integralcost} is replaced by \eqref{eq:ff} with $f^\prime$ injective\footnote{The condition that $f^\prime$ is injective on the spectrum of $A_t$ can be guaranteed from the outset, if it is assumed strictly monotonic.} on the spectrum of $A_t$, is isospectral.
\end{proposition} 

\begin{remark}
    In Proposition \ref{prop:6}, injectivity of $f^\prime$ is assumed on the spectrum of $A_t$ which, with $A_t$ the solution to an optimization problem, is not known apriori. Global injectivity of $f^\prime$ is sufficient.
\end{remark}
\begin{remark}\label{rem:kalman}
 We propose in passing the {\em inverse problem} of isospectral control, in the spirit of \cite{kalman1964linear}, to determine the range of the map
 \[
 (\Sigma_0,\Sigma_T,f(\cdot))\mapsto \mathrm{spec}(A_t)
 \]
 over choices of a function $f(\cdot)$ in \eqref{eq:ff}.
 In other words, we would like to know what choices of eigenvalues in $D$ in setting up Problem \ref{prob1} may arise as solutions to Problem \ref{prob2}, for a suitable choice of $f$. In Kalman's spirit, given an isospectral $A_\cdot$, {\em how can we tell that it is optimal for a suitable functional}? This question remains open at present.
\end{remark}

\section{Conditions for Isospectral Controllability}\label{isospcon}
We briefly recall {\em vector majorization}, a concept that is central to the results of this section.
Majorization represents a partial order in $\mathbb R^n$ that quantifies in a certain sense the degree of non-uniformity among the vector-entries; intuitively $\mathbf x\in\mathbb R^n$ is majorized by $\mathbf y\in\mathbb R^n$ if the sum of their entries is the same and, individually, the entries of $\mathbf x$ are  ``closer to each other'' than those of $\mathbf y$. More precisely, $\mathbf y$ majorizes $\mathbf x$ if $\mathbf x$ can be obtained from $\mathbf y$ by a finite sequence of ``Robin Hood'' redistribution of ``wealth''---taking from larger components of $\mathbf y$ and redistributing to smaller ones \cite{MOA11}.
Formally,
with $\mathbf x,\mathbf y\in\mathbb R^n$, let $\mathbf x^\downarrow,\mathbf y^\downarrow$ denote column vectors with entries equal to those of $\mathbf x,\mathbf y$ rearranged in decreasing order;
i.e., such that $x^\downarrow_1 \ge x^\downarrow_2 \ge \dots \ge x^\downarrow_n$ and similarly for $\mathbf y^\downarrow$. We say that $\mathbf x$ is {\em majorized} by $\mathbf y$ and denote by $\mathbf x \prec \mathbf y$, if
$\sum_{i=1}^n x_i = \sum_{i=1}^n y_i$ and for all $k\in\{1,\ldots, n-1\}$, 
    \begin{align*}
    \sum_{i=1}^k x^\downarrow_i &\le \sum_{i=1}^k y^\downarrow_i.
    \end{align*}
    An alternative characterization is given by a theorem of Hardy--Littlewood--P\'olya that $\mathbf x\prec \mathbf y$ if and only if 
    $\mathbf x=P \mathbf y$ for some doubly stochastic\footnote{A matrix with non-negative entries and column/row sums equal to $1$.} matrix $P$.
    
    We next state the main necessary condition for a terminal covariance $\Sigma_T>0$ to be reachable from $\Sigma_0$ as sought in Problem \ref{prob1}. The proof, which relies on tools from multilinear algebra, is given in Section~\ref{sec:proof-main} after the necessary background is developed in Section~\ref{sec:background}.
\begin{theorem}\label{maintheorem}
    If $\Sigma_{T}\in\mathcal{R}(\Sigma_0,D,T)$, then it holds that\footnote{Note that in the right hand side of \eqref{eq:majorization}, the two vectors are added after being sorted in decreasing order. It is this vector that majorizes $\log{\blambda_T}$.} 
    \begin{align}\label{eq:majorization}
       \log{\blambda_T}\prec \log{\blambda_0^\downarrow}+2T{\blambda_D^\downarrow},
    \end{align}
    with  $\log{\blambda_{T }}$, $\log{\blambda_0}$, and ${\blambda_D}$ be the vectors of eigenvalues of $\log\Sigma_{T}$, $\log\Sigma_0$, and $D$, respectively. 
\end{theorem}
 
Before we proceed with the proof of Theorem \ref{maintheorem}, we present a case where the  majorization condition in the theorem is also sufficient, thereby fully characterizing the reachable set solving Problem \ref{prob1}. 
 
\begin{theorem}\label{thm:2}
   Let $\mathbf{\blambda_D}$ be the vector of eigenvalues of $D$. If $\Sigma_0=cI,~c>0$, then 
    \begin{align*}
      \mathcal{R}(cI,D,T)= \{\Sigma_T \in \mathbb{S}_{++}^n~:~\log\mathbf{\blambda_T}\prec (\log c)\mathbf{1}+2T\mathbf{\blambda_D}\},
    \end{align*}
    where $\log\mathbf{\blambda_{T }}$ is the vector of eigenvalues of $\log\Sigma_{T}$ and $\mathbf{1}$ is the vector of ones.
\end{theorem}
\vspace{3mm}
\begin{proof}
   By the Hardy-Littlewood-P\'olya Theorem \cite[Theorem B.2]{MOA11}, the majorization
    $$
\log\mathbf{\blambda_T}\prec (\log c)\mathbf{1}+2T\mathbf{\blambda_D},$$
 is equivalent to the existence of a bistochastic matrix $P$ such that 
\begin{align*}
    \log\mathbf{\blambda_T}&=P((\log c )\mathbf{1}+2T\mathbf{\blambda_D})\\
    &=(\log c) \mathbf{1}+2TP\mathbf{\blambda_D}.
\end{align*}
By the Birkhoff-von Neumann Theorem \cite[Page 30, Theorem A.2]{MOA11}, any bistochastic matrix $P$ lies within the convex hull of permutation matrices. Thus,  $\log{\blambda_T}$ can be obtained from a convex combination of permuted copies of $2T{\blambda_D}$, i.e., there exist coefficients $\alpha_\pi$ such that 
\begin{align*}
    &\log{\blambda_T}  = (\log c)\mathbf{1}+2T\sum_{\pi\in S_n}\alpha_\pi \Pi_\pi \mathbf{\blambda_D},\\
    &\alpha_\pi\geq 0,~~\sum_{\pi\in S_n}\alpha_\pi = 1,
\end{align*}
where $S_n$ is the symmetric group of permutations and $\Pi_\pi$ are the corresponding permutation matrices; note that permutation matrices are also orthogonal, i.e., $\Pi_\pi\in O(n)$. 
Consider again the bilinear system 
\[
  \dot{X}_t = A_t X_t, \qquad X_0 = \sqrt{c}I.
\]
By partitioning the time-interval $[0,T]$ into disjoint union of subsets $I_\pi$ of length $\alpha_\pi T$, namely,
\[
[0,T] =\bigcup_\pi I_\pi, \mbox{ with }I_\pi\cap I_{\pi^\prime}=\varnothing \mbox{ when }\pi\neq \pi^\prime,
\]
and defining the  piecewise-constant control 
$$
A_t = W\Pi_\pi D \Pi_\pi^\top W^\top, \mbox{ for }t\in I_\pi,
$$ 
where $W\in O(n)$ diagonalizes $\Sigma_T$. Starting from $X_0=\sqrt{c}I$ we get\footnote{Note that the commutator $[A_t,A_s]=0$ for any $t,s\in[0,T]$, because $\Pi_\pi D\Pi_\pi^\top$ is diagonal. This justifies the unordered products, since all matrices that arise commute.}
\begin{align*}
  X_{T}
  &= \prod_{\pi\in S_n} e^{\,\alpha_\pi T W\Pi_\pi D \Pi_\pi^\top W^\top}X_0,\\
  &= e^{\sum_{\pi\in S_n}\alpha_\pi T W\Pi_\pi D \Pi_\pi^\top W^\top}e^{\frac{1}{2}(\log c)I}\\
  &= We^{\frac{1}{2}(\log c)I+\sum_{\pi\in S_n}\alpha_\pi T \Pi_\pi D \Pi_\pi^\top }W^\top\\
  &= We^{\frac12\diag(\log(\blambda_T))}W^\top\\
  &= W(\text{diag}(\mathbf{\blambda_T}))^{1/2}W^\top = \Sigma_T^{1/2}.   
\end{align*}
Thus $A_t$ satisfies \eqref{eq:control} and steers $\Sigma_0$ into $\Sigma_T$ via \eqref{eq:lyap}, proving that $\Sigma_T$ is indeed reachable as claimed.   
\end{proof}  
\section{Background and Technical Lemmas}\label{sec:background}
 
The proof of Theorem~\ref{maintheorem} relies on multilinear algebra and, specifically, the theory of exterior powers and compound matrices. We provide a self-contained development of the necessary background and collect the needed key technical results in this section.
 
Two distinct ``compound'' operations appear throughout that are listed next for ease of reference \cite{flanders,MOA11}.
\begin{itemize}
\item[-] The \emph{multiplicative compound}
    (or \emph{$k$-th exterior power}) of a matrix $X$ is denoted by
    $\wedge^k\!X$. It is the $\binom{n}{k}\times\binom{n}{k}$ matrix whose entries are
    the $k\times k$ minors of $X$.
\item[-] The \emph{additive compound} of a matrix $A$ is denoted by
    $A^{[k]}$. It is the derivative of the multiplicative compound of $e^{tA}$
    at $t=0$:
    $A^{[k]}=\frac{d}{dt}\big|_{t=0}\wedge^k(e^{tA})$.
\end{itemize}
The relationship between the two is analogous to
that between a Lie group element and the corresponding
Lie algebra element: $\wedge^k\!X$ describes a finite
transformation on $k$-vectors, while $A^{[k]}$ describes its
infinitesimal generator. When both appear in the same formula, as
in the compound ODE of Lemma~\ref{lem:compound_ode}, $A_t^{[k]}$ is the state matrix of the differential equation.
 
\subsection{Exterior powers and wedge products}\label{sec:wedge}
 
Given two vectors $\mathbf  u,\mathbf v\in\mathbb{R}^n$, one seeks a rigorous algebraic representation of the ``oriented area'' of the parallelogram they span. In $\mathbb{R}^3$, the magnitude of this area is traditionally represented indirectly by the normal vector generated by the cross product $\bf u\times v$. The \emph{wedge product} $\bf u\wedge v$ captures this oriented area directly as a distinct geometric object called a \emph{bivector}, providing a framework that naturally scales to any dimension. Its key property is \emph{antisymmetry}:
\[
\mathbf  u\wedge \mathbf v = -\,\mathbf v\wedge \mathbf u,
\]
which encodes orientation, i.e., swapping the two vectors flips the sign. In particular, ${\bf u\wedge u} = 0$: a parallelogram with two identical sides has zero area.
 
More generally, for $k$ vectors ${ \mathbf v}_1,\ldots,{\mathbf v}_k\in\mathbb{R}^n$, the wedge
product ${\bf v}_1\wedge\cdots\wedge {\bf v}_k$ represents the oriented $k$-dimensional
volume of the parallelepiped they span. It changes sign under any
transposition of two vectors, and vanishes whenever two vectors coincide
(or, more generally, whenever ${\mathbf  v}_1,\ldots,{\mathbf v}_k$ are linearly dependent).
 
Now let $\mathbf  e_1,\ldots,\mathbf e_n$ be the standard basis of $\mathbb{R}^n$. The
\emph{$k$-th exterior power} $\wedge^k\mathbb{R}^n$ is the vector space
spanned by elements
\[
  {\mathbf  e_I :=\mathbf e_{i_1}\wedge \mathbf e_{i_2}\wedge\cdots\wedge \mathbf e_{i_k}},
  \qquad 1\le i_1 < i_2 < \cdots < i_k \le n,
\]
where $I=(i_1,\ldots,i_k)$ is an ordered multi-index.
There are $\binom{n}{k}$ such basis
elements, and so $\dim(\wedge^k\mathbb{R}^n)=\binom{n}{k}$.
 
For any ${\mathbf v_1,\ldots,\mathbf v_k}\in\mathbb{R}^n$, with
${\mathbf v_j} = \sum_{i=1}^n v_{ij}{\mathbf e_i}$, the wedge product ${\mathbf v_1\wedge\cdots\wedge \mathbf v_k}$ is given by 
\begin{equation}\label{eq:wedge_coords}
  \sum_{1\le i_1<\cdots<i_k\le n}
    \det\bigl(v_{i_m\ell}\bigr)_{1\le m,\ell\le k}\;{\mathbf e_{i_1}\wedge\cdots
    \wedge \mathbf e_{i_k}}.
\end{equation}
That is, the coordinate of ${\mathbf v_1\wedge\cdots\wedge \mathbf v_k}$ along ${\mathbf e_I}$ is the $k\times k$ minor of the $n\times k$ matrix
$[v_{ij}]_{i,j=1}^{n,k}$, with the rows selected according to $I$.
When $k=n$, the space $\wedge^n\mathbb{R}^n$ is one-dimensional and
$\mathbf v_1\wedge\cdots\wedge \mathbf v_n = \det([v_{ij}])\;\mathbf e_1\wedge\cdots\wedge \mathbf e_n$.
 
The standard inner product on $\mathbb{R}^n$ induces one on
$\wedge^k\mathbb{R}^n$ by declaring the basis $\{e_I\}$ to be orthonormal:
$\langle \mathbf e_I,\,\mathbf e_J\rangle = \delta_{IJ}$, the Dirac delta.
 
\begin{example}[$n=3,\;k=2$]\label{ex:wedge32}
Let $\mathbf u=(u_1,u_2,u_3)^\top$ and $\mathbf v=(v_1,v_2,v_3)^\top$. The wedge product is
\[
  \mathbf u\wedge \mathbf v
  = \begin{vmatrix}
      u_1& v_1\\
      u_2& v_2
  \end{vmatrix}\,\mathbf e_1\wedge \mathbf e_2
  + \begin{vmatrix}
      u_1& v_1\\
      u_3& v_3
  \end{vmatrix}\,\mathbf e_1\wedge \mathbf e_3
  + \begin{vmatrix}
      u_2& v_2\\
      u_3& v_3
  \end{vmatrix}\,\mathbf e_2\wedge \mathbf e_3.
\]
The three components are the $2\times 2$ minors of the $3\times 2$ matrix $(\mathbf u|\mathbf v)$, and the result is a vector in $\wedge^2\mathbb{R}^3\cong\mathbb{R}^3$. In $\mathbb{R}^3$, this is equivalent to the cross product $\mathbf u\times \mathbf v$.
\end{example}
 
\subsection{The multiplicative compound $\wedge^k\!X$}\label{sec:mult_compound}
 
A linear map $X:\mathbb{R}^n\to\mathbb{R}^n$ induces a linear map
$\wedge^k\! X:\wedge^k\mathbb{R}^n\to\wedge^k\mathbb{R}^n$ defined by
\begin{equation}\label{eq:induced}
  (\wedge^k\! X)(\mathbf v_1\wedge\cdots\wedge \mathbf v_k)
  := (X\mathbf v_1)\wedge\cdots\wedge(X\mathbf v_k),
\end{equation}
extended to all of $\wedge^k\mathbb{R}^n$ by linearity.
The matrix representation of $\wedge^k\! X$ in the basis $\{\mathbf e_I\}$ has entries
\begin{equation}\label{eq:compound_entries}
  (\wedge^k\! X)_{IJ}
  = \det\bigl(X_{i_\ell,j_m}\bigr)_{1\le\ell,m\le k},
\end{equation}
the $k\times k$ minor of $X$ with rows $I$ and columns $J$.
 
\begin{example}[$n=3,\;k=2$]\label{ex:mult_compound}
For $X=(X_{ij})_{3\times 3}$, the multiplicative compound $\wedge^2 X$ is a $3\times 3$ matrix with rows and columns indexed by $(1,2),(1,3),(2,3)$:
\[
  \wedge^2\! X =
  \begin{pmatrix}
   \begin{vmatrix}
       X_{11}& X_{12}\\ X_{21}& X_{22}
   \end{vmatrix}  & \begin{vmatrix}
       X_{11}& X_{13}\\ X_{21}& X_{23}
   \end{vmatrix} &
    \begin{vmatrix}
       X_{12}& X_{13}\\ X_{22}& X_{23}
   \end{vmatrix}\\[11pt]
     \begin{vmatrix}
       X_{11}& X_{12}\\ X_{31}& X_{32}
   \end{vmatrix} & \begin{vmatrix}
       X_{11}& X_{13}\\ X_{31}& X_{33}
   \end{vmatrix} &
      \begin{vmatrix}
       X_{12}& X_{13}\\ X_{32}& X_{33}
   \end{vmatrix}\\[11pt]
   \begin{vmatrix}
       X_{21}& X_{22}\\ X_{31}& X_{32}
   \end{vmatrix} &  \begin{vmatrix}
       X_{21}& X_{23}\\ X_{31}& X_{33}
   \end{vmatrix} &
      \begin{vmatrix}
       X_{22}& X_{23}\\ X_{32}& X_{33}
   \end{vmatrix}
  \end{pmatrix}.
\]
The boundary cases are  $\wedge^1 X = X$ and $\wedge^n X = (\det X)$, the $1\times 1$ matrix consisting of the determinant.
\end{example}
 
The multiplicative compound satisfies three fundamental algebraic properties:
\begin{itemize}
    \item[-] \emph{Multiplicative}: $\wedge^k(XY) = (\wedge^k\! X)(\wedge^k\! Y)$.
    \footnote{Follows from applying \eqref{eq:induced} to $XY$: $(XYv_1)\wedge\cdots\wedge(XYv_k) = (\wedge^k\! X)((Yv_1)\wedge\cdots\wedge(Yv_k))$.}
    \item[-] \emph{Transposition}: $\wedge^k(X^\top) = (\wedge^k\! X)^\top$.  \footnote{Follows from \eqref{eq:compound_entries} and the transpose symmetry of minors.} 
    \item[-] \emph{Identity}: $\wedge^k I_n = I_{\binom{n}{k}}$, the $\binom{n}{k}\times \binom{n}{k}$ identity matrix.
\end{itemize}
It follows that if $A$ is symmetric/orthogonal, then so is $\wedge^k\! A$.
 
\begin{remark}
    Ultimately, the multiplicative compound map $\wedge^k$ should be thought of as a functor, called the $\textit{k-th exterior power functor}$, from the category of finite-dimensional vector spaces to itself. It simply takes the objects (vector spaces $V$) to their $k$-th exterior powers ($\Lambda^k V$), and the morphisms (matrices $X$) to their $k$-th compounds $\Lambda^kX$. Many properties enjoyed by $X$ are automatically inherited by $\wedge ^k$. These include invertibility, symmetry, orthogonality, and transposition. Most crucially, if one knows the eigenvalues and eigenvectors of $X$, then we immediately know what they are for $\wedge^k X$. The eigenvectors are simply wedge products of those of $X$, and the eigenvalues are simply the \textit{products} of those of $X$. Another interesting property is that of rank. If the original matrix $X$ has rank $r$, then the rank of $\wedge^k X$ is ${r\choose k}$. If $k$ is strictly greater than $r$, the compound matrix is set to be the zero matrix.
    Note that functors satisfy two basic properties: preservation of the identity and of composition. The latter \textit{is} the celebrated Cauchy-Binet formula \cite{flanders}. The functorial approach in multilinear algebra can often help elegantly bypass nightmarish computations/proofs involving lots of indices and symbols!
\end{remark}

\subsection{The additive compound $A^{[k]}$}\label{sec:add_compound}
 
While $\wedge^k\! X$ describes how a finite transformation acts on $k$-volumes, the {additive compound} $A^{[k]}$ describes how an {infinitesimal} transformation acts. Formally, it is the derivative of the multiplicative compound at the identity:
\begin{equation}\label{eq:add_compound_def}
  A^{[k]} :=  \frac{d}{dt}\Bigr|_{t=0}\wedge^k(e^{tA}),
\end{equation}
Note that $e^{tA}$ in \eqref{eq:add_compound_def} can be replaced by $(I_n+tA)$. Expanding $\wedge^k(I_n+tA)$ to first order in $t$, each entry is a $k\times k$ minor of $I_n+tA$. The $O(1)$ term is $\delta_{IJ}$ and the $O(t)$ term retains exactly one factor of $A$ while the remaining $k-1$ factors come from $I$. This yields the explicit formula:
 
\begin{itemize}
\item \emph{Diagonal entries}: $A^{[k]}_{II} = \sum_{\ell=1}^k A_{i_\ell,i_\ell}$, the sum of the diagonal entries of $A$ indexed by $I$.
\item \emph{Off-diagonal entries}: If $I$ and $J$ differ in exactly one element ($i_r\neq j_s$, with all other elements shared), then $A^{[k]}_{IJ} = (-1)^{r+s}A_{i_r,j_s}$. If $I$ and $J$ differ in two or more elements, $A^{[k]}_{IJ} = 0$.
\end{itemize}
 
\begin{example}[$n=3,\;k=2$]\label{ex:add_compound_3}
With the lexicographic basis order $(1,2),\,(1,3),\,(2,3)$, the additive compound of $A\in\mathbb{R}^{3\times 3}$ is the $3\times 3$ matrix
\[
  A^{[2]} =
  \begin{pmatrix}
    A_{11}+A_{22} & A_{23} & -A_{13}\\[3pt]
    A_{32} & A_{11}+A_{33} & A_{12}\\[3pt]
    -A_{31} & A_{21} & A_{22}+A_{33}
  \end{pmatrix}.
\]
To verify the off-diagonal signs: indices $(1,2)$ and $(1,3)$ differ only in the second position ($2\to 3$, so $r=s=2$), giving $(-1)^{2+2}A_{23}=A_{23}$. Indices $(1,2)$ and $(2,3)$ differ in the first position of $I$ and the second of $J$ ($1\to 3$, $r=1$, $s=2$), giving $(-1)^{1+2}A_{13}=-A_{13}$.
\end{example}
 
\begin{example}[Boundary cases]\label{ex:boundary}
For $k=1$: $A^{[1]}=A$. For $k=n$: $A^{[n]}=(\tr A)$, the $1\times 1$ matrix, i.e., a scalar, consisting of the trace; this is consistent with $\frac{d}{dt}\det(e^{tA})\big|_{t=0} = \tr A$.
\end{example}
 
The key distinction between the two compounds is captured by the exponential relationship
\begin{equation}\label{eq:exp_relation}
\wedge^k(e^{tA}) = e^{tA^{[k]}},
\end{equation}
which follows from the multiplicativity of $\wedge^k$ and the definition \eqref{eq:add_compound_def}. Equation \eqref{eq:exp_relation} is the precise sense in which $A^{[k]}$ generates $\wedge^k(e^{tA})$.
 
\subsection{Technical lemmas}
 
We now state and prove the four lemmas that underpin the proof of Theorem~\ref{maintheorem}. The first three concern the algebraic structure of compound matrices, while the fourth provides an analytic growth estimate.
 
\begin{lemma}[Singular values of the multiplicative compound]\label{lem:compound_svd}
Let $X = U\diag(\sigma_1,\ldots,\sigma_n)V^\top$ be the SVD of
$X\in\mathbb{R}^{n\times n}$ with $\sigma_1\ge\cdots\ge\sigma_n\ge 0$. Then
\[
  \wedge^k\! X
  = (\wedge^k\! U)\,\diag\Bigl(\sigma_{i_1}\cdots\sigma_{i_k}
    \Bigr)_{1\le i_1<\cdots<i_k\le n}\,(\wedge^k\! V)^\top.
\]
In particular, the singular values of $\wedge^k\! X$ are the $\binom{n}{k}$
products $\sigma_{i_1}\cdots\sigma_{i_k}$, and the largest singular
value is
\[
  \|\wedge^k\! X\| = \prod_{i=1}^k\sigma_i(X).
\]
\end{lemma}
 
\begin{proof}
By multiplicativity and the transpose property of $\wedge^k$,
\[
  \wedge^k\! X
  = \wedge^k(U\Sigma V^\top)
  = (\wedge^k U)\,(\wedge^k \Sigma)\,(\wedge^k V)^\top,
\]
where $\Sigma = \diag(\sigma_1,\ldots,\sigma_n)$. Since $U$ and $V$ are
orthogonal, $(\wedge^k U)^\top (\wedge^k U) = \wedge^k(U^\top U) = \wedge^k I_n =
I_{\binom{n}{k}}$, so $\wedge^k U$ and $\wedge^k V$ are
orthogonal.
Since $\Sigma$ is diagonal, the $k\times k$ minor with rows $I$ and
columns $J$ vanishes unless $I=J$ (any off-diagonal minor of a diagonal matrix has a zero row or column), and equals $\sigma_{i_1}\cdots\sigma_{i_k}$ when $I=J$.
Therefore $\wedge^k \Sigma = \diag(\sigma_{i_1}\cdots\sigma_{i_k})$. This is a valid SVD (orthogonal factors, non-negative diagonal), and the largest diagonal entry is
$\sigma_1\cdots\sigma_k$.
\end{proof}

\begin{lemma}[Eigenvalues of the additive compound]\label{lem:compound_eig}
If $A$ is symmetric with eigenvalues $a_1\ge\cdots\ge a_n$, then $A^{[k]}$
is symmetric with eigenvalues
\[
  a_{i_1}+\cdots+a_{i_k},\qquad 1\le i_1<\cdots<i_k\le n,
\]
and largest eigenvalue $\sum_{i=1}^k a_i$.
\end{lemma}
 
\begin{proof}
Symmetry of $A^{[k]}$ follows from the transpose property: if $A=A^\top$, then $\wedge^k(I+tA)$ is symmetric for every $t$, and hence so is its derivative $A^{[k]}$ at $t=0$.
 
For the eigenvalues, let $A\mathbf w_j = a_j\mathbf w_j$ with
$\{\mathbf w_j\}$ orthonormal, and consider the element
$\mathbf w_I := \mathbf w_{i_1}\wedge\cdots\wedge \mathbf w_{i_k}\in\wedge^k\mathbb{R}^n$.
By the product rule applied to \eqref{eq:induced}:
\begin{align*}
  A^{[k]}\mathbf w_I
  &=\frac{d}{dt}\Bigr|_{t=0}\wedge^k(I+tA)\,\mathbf w_I\\
  &= \frac{d}{dt}\Bigr|_{t=0}(\mathbf w_{i_1}+tA\mathbf w_{i_1})\wedge\cdots\wedge(\mathbf w_{i_k}+tA\mathbf w_{i_k})\\
  &= \sum_{\ell=1}^k
     \mathbf w_{i_1}\wedge\cdots\wedge\underset{\ell\text{-th}}{(A\mathbf w_{i_\ell})}
     \wedge\cdots\wedge \mathbf w_{i_k}\\
  &= \sum_{\ell=1}^k a_{i_\ell}\;
     \mathbf w_{i_1}\wedge\cdots\wedge \mathbf w_{i_k}
  = \Bigl(\sum_{\ell=1}^k a_{i_\ell}\Bigr)\mathbf w_I.
\end{align*}
The key step is the fourth equality, where we use $A\mathbf w_{i_\ell}=a_{i_\ell}\mathbf w_{i_\ell}$ and the fact that replacing $\mathbf w_{i_\ell}$ by its scalar multiple does not change the other factors. Since $\{\mathbf w_I\}$ is an orthonormal basis for $\wedge^k\mathbb{R}^n$ (as $\{\mathbf w_j\}$ is for $\mathbb{R}^n$), we have identified all $\binom{n}{k}$ eigenvalues.
\end{proof}
 
\begin{example}
Consider $A=\diag(3,1,-4)$ with $n=3$, $k=2$. By Lemma~\ref{lem:compound_eig}, the eigenvalues of $A^{[2]}$ are the three pairwise sums: $3+1=4$, $3+(-4)=-1$, $1+(-4)=-3$. One may verify this directly from Example~\ref{ex:add_compound_3}:
\[
A^{[2]}=\begin{pmatrix}4 & 0 & 0\\ 0 & -1 & 0\\ 0 & 0 & -3\end{pmatrix},
\]
which is diagonal with the claimed eigenvalues. For $k=3$: $A^{[3]}=(\tr A)=(0)$.
\end{example}

\begin{lemma}[Compound ODE]\label{lem:compound_ode}
If $X_t$ solves $\dot X_t = A_t X_t$, then the multiplicative compound $\wedge^k\!X_t$ solves
\begin{equation}\label{eq:compound_ode}
  \frac{d}{dt}(\wedge^k\! X_t) = A_t^{[k]}\,(\wedge^k\! X_t), \qquad \wedge^k\! X_0 = \wedge^k\!(X_0).
\end{equation}
That is, the multiplicative compound evolves under the additive compound of the coefficient.
\end{lemma}
 
\begin{proof}
For constant $A$, $X_t = e^{tA}$ and by the exponential relation \eqref{eq:exp_relation},
$\wedge^k\! X_t = \wedge^k(e^{tA}) = e^{tA^{[k]}}$.
Differentiating:
\begin{align*}
  \frac{d}{dt}\wedge^k(e^{tA})
  &= \frac{d}{dt}e^{tA^{[k]}}
  = A^{[k]}\cdot e^{tA^{[k]}}
  = A^{[k]}\cdot(\wedge^k\! X_t).
\end{align*}
For time-varying $A_t$, the same conclusion follows from the multiplicativity of $\wedge^k$. For small $h$,
$X_{t+h}\approx(I+hA_t)X_t$, so
\[
  \wedge^k\! X_{t+h}\approx\wedge^k(I+hA_t)\cdot(\wedge^k\! X_t)
  \approx(I+hA_t^{[k]})\,(\wedge^k\! X_t).
\]
Subtracting $\wedge^k\! X_t$, dividing by $h$, and taking $h\to 0$ gives $\frac{d}{dt}(\wedge^k\! X_t)=A_t^{[k]}\,(\wedge^k\! X_t)$.
\end{proof}

\begin{lemma}[Operator norm growth]\label{lem:norm_growth}
Let $Z$ be an absolutely continuous function $Z:[0,T ]\to\mathbb R^{n\times n}$ that satisfies
\[
\dot Z_t=B_tZ_t
~~\text{for a.e. }t\in[0,T ],~~\det(Z_0) \neq 0,
\]
where $B:[0,T ]\to\mathbb R^{n\times n}$ is symmetric and integrable. Then
\[
\frac{d}{dt}\log\|Z_t\|\le \lambda_{\max}(B_t)
~~\text{for a.e. }t\in[0,T ],
\]
with $\|Z_t\|>0$ on $[0,T ]$.
\end{lemma}
\vspace{2mm}
\begin{proof}
 By Liouville's formula,
\[
\det Z_t=\det Z_0\exp\!\left(\int_0^t \operatorname{tr}(B_\tau)\,d\tau\right)\neq 0,~~~t\in[0,T ].
\]
Hence $Z_t$ is invertible for every $t\in[0,T ]$, and therefore $\|Z_t \mathbf v\|>0$
for every $\mathbf v\neq 0$. Now fix a unit vector  $\mathbf v\in\mathbb R^n$. Since $Z$ is absolutely continuous, the map
$t\mapsto \|Z_t \mathbf v\|^2$ is absolutely continuous, and for a.e.\ $t$,
\begin{align*}
\frac{d}{dt}\|Z_t \mathbf v\|^2
&=
2\langle \dot Z_t \mathbf v,Z_t \mathbf v\rangle\\
&=
2\langle B_tZ_t \mathbf v,Z_t \mathbf v\rangle
\le
2\lambda_{\max}(B_t)\,\|Z_t \mathbf v\|^2,
\end{align*}
where the last inequality uses the symmetry of $B_t$: $\langle B_t \mathbf w,\mathbf w\rangle \le \lambda_{\max}(B_t)\|\mathbf w\|^2$ for any $\mathbf w$. Dividing by $2\|Z_t \mathbf v\|^2$ and using the chain rule yields
\[
\frac{d}{dt}\log\|Z_t \mathbf v\|
\le \lambda_{\max}(B_t)
~~\text{for a.e. }t\in[0,T ].
\]
Integrating from $s$ to $t$, $0\le s\le t\le T $, and exponentiating gives
\[
\|Z_t \mathbf v\|
\le
e^{\int_s^t \lambda_{\max}(B_\tau)\,d\tau}\,\|Z_s \mathbf v\|.
\]
Taking $\mathbf v^\star$ of unit norm to maximize $\|Z_t\mathbf v\|$,
\[
\|Z_t\|
\le
e^{\int_s^t \lambda_{\max}(B_\tau)\,d\tau}\,\|Z_s\mathbf v^\star\|\leq
e^{\int_s^t \lambda_{\max}(B_\tau)\,d\tau}\,\|Z_s\|.
\]
Taking logarithms, we obtain
\[
\log\|Z_t\|-\log\|Z_s\|
\le
\int_s^t \lambda_{\max}(B_\tau)\,d\tau.
\]
Since $t\mapsto \log\|Z_t\|$ is absolutely continuous, so is
\[
t\mapsto \log\|Z_t\|-\int_0^t \lambda_{\max}(B_\tau)\,d\tau.
\]
Moreover it is non-increasing, implying that
\begin{align*}
\frac{d}{dt}\log\|Z_t\|\le \lambda_{\max}(B_t)
~~\text{for a.e. }t\in[0,T ].
\end{align*}
This completes the proof.
\end{proof}
 
 \section{Proof of Theorem~\ref{maintheorem}}\label{sec:proof-main}
 
We are now in a position to prove the main result, drawing on the four lemmas established in Section~\ref{sec:background}.
 
\begin{proof}
   Take any $\Sigma_T \in\mathcal{R}({\Sigma_0,D,T })$, and a corresponding $A_t$ that satisfies \eqref{eq:control} and steers $\Sigma_0$ into $\Sigma_T$ via \eqref{eq:lyap}. Letting
   $X_0=\Sigma_0^{1/2}$
   and $X_t$ the solution of
\[
  \dot{X}_t = A_t X_t,
\]
it follows that
\begin{equation}\label{eq:SXX}
\Sigma_t=X_tX_t^\top
\end{equation}
and hence, that $X_{T}=\Sigma_T^{1/2}V$ for some orthogonal matrix $V$. Fix $k\in\{1,\dots,n\}$ and consider the induced dynamics on the $k$-th exterior
power $\wedge^k\mathbb{R}^n$. The multiplicative compound $\wedge^k\!X_t$ evolves under the additive compound $A_t^{[k]}$: by Lemma~\ref{lem:compound_ode},
\begin{equation}\label{eq:diffwedge}
  \frac{d}{dt}(\wedge ^k\! X_t) = A_t^{[k]}\,(\wedge ^k\! X_t), \qquad \wedge ^k\! X_0 = \wedge^k\!\Sigma_0^{1/2}.
\end{equation}
Since $A_t$ is symmetric, Lemma~\ref{lem:compound_eig} implies that the additive compound $A_t^{[k]}$ is also symmetric, with eigenvalues given by the $\binom{n}{k}$ sums
$\lambda_{i_1}(D)+\cdots+ \lambda_{i_k}(D)$ for $1\le i_1<\cdots<i_k\le n$.
Since $ \lambda_1(D)\ge\cdots\ge  \lambda_n(D)$, the largest eigenvalue is
\[
  \lambda_{\max}(A_t^{[k]}) = \sum_{i=1}^k \lambda_i(D).
\]
By Lemma~\ref{lem:compound_svd}, the operator norm of the multiplicative compound $\wedge^k\! X_t$ satisfies
$$\|\wedge^k\! X_t\| = \prod_{i=1}^k\sigma_i(X_t),$$ where
$\sigma_1(X_t)\ge\cdots\ge\sigma_n(X_t)$ are the singular values of $X_t$.
Applying Lemma~\ref{lem:norm_growth} to the compound ODE \eqref{eq:diffwedge}---noting that the coefficient $A_t^{[k]}$ is symmetric---we obtain
\begin{align*}
  \frac{d}{dt}\log\|\wedge^k\! X_t\|
  \le \lambda_{\max}(A_t^{[k]})
  = \sum_{i=1}^k \lambda_i(D),
\end{align*}
for a.e.\ $t\in[0,T]$ and $k\in\{1,\cdots,n\}$. Integrating from $0$ to $T$ yields
\begin{equation}\label{eq:partial_ineq}
  \sum_{i=1}^k\left(\log\sigma_i(X_T)-\log\sigma_i(X_0)\right) \le T\sum_{i=1}^k \lambda_i(D).
\end{equation}
Rearranging and using the fact that
$\sigma_i^2(X_t)= \lambda_i(\Sigma_t)$
from \eqref{eq:SXX}, we obtain
\begin{subequations}
\begin{equation}\label{eq:partial_ineq2}
  \sum_{i=1}^k\log\lambda_i(\Sigma_T)\le \sum_{i=1}^k \left(\log\lambda_i(\Sigma_0)+2T\lambda_i(D)\right),
\end{equation}
for $k\in\{1,\ldots,n\}$.
At $k=n$, Liouville's formula gives equality, namely
\begin{align}\nonumber
  \log\det \Sigma_T
  &=  \log\det \Sigma_0+2\int_0^{T}
  \tr(A_t)\,dt\\\label{eq:liouville}
  &= \log\det \Sigma_0+2T\tr(D) .
\end{align}
\end{subequations}
Combining (\ref{eq:partial_ineq2}--\ref{eq:liouville}) we obtain the claimed majorization \begin{align*}
       \log{\blambda_T}\prec \log{\blambda_0^\downarrow}+2T{\blambda_D^\downarrow},
    \end{align*}
    which completes the proof.
\end{proof}
 
\section{Time-reversal and isospectral controllability}\label{sec:time-reversal}
The reachable set $\mathcal{R}(\Sigma_0,D,T)$  defined in Problem \ref{prob1} enjoys a time-symmetry property. Namely, reaching $\Sigma_T$ from $\Sigma_0$ with eigenvalues specified in $D$ is equivalent to reaching $\Sigma_0$ from $\Sigma_T$ with $-D$. This allows us to obtain additional conditions for controllability by taking advantage of the results in Section \ref{isospcon}. This is explored next.
\begin{proposition}\label{prop12}
    Let $\Sigma_0,\Sigma_{T}>0$, $T>0$, and $D$ given, then the following are equivalent,
    \begin{subequations}
    \begin{align*}
    i) \;&\Sigma_{T}\in\mathcal{R}(\Sigma_0,D,T)\\
    ii)\; &\Sigma_0\in\mathcal{R}(\Sigma_{T},-D,T).
    \end{align*}
    \end{subequations}
\end{proposition}
\begin{proof}
    It is enough to prove one implication since the reverse holds by the symmetry of the statement. Suppose $\Sigma_T\in\mathcal{R}(\Sigma_0,D,T)$, then by definition of $\mathcal{R}(\Sigma_0,D,T)$ there exists a control $A_t=U_tDU_t^\top$ taking $\Sigma_0$ to $\Sigma_{T}$ via
    \begin{align*}
       \dot\Sigma_t = A_t\Sigma_t+\Sigma_tA_t.
    \end{align*}
    The time-reversed trajectory $\tilde\Sigma_t:=\Sigma_{T-t}$ satisfies 
    \begin{align*}
        \dot{\tilde{\Sigma}}_t = \dot{\Sigma}_{T-t}&=-A_{T-t}\Sigma_{T-t}-\Sigma_{T-t}A_{T-t}\\
        &=\tilde{A}_t\tilde{\Sigma}_t+\tilde{\Sigma}_t\tilde{A}_t,
    \end{align*}
    where $\tilde{A}_t:=-A_{T-t}$. Thus we have found a new control $\tilde{A}_t=-U_{T-t}DU_{T-t}^\top$ with spectrum in $-D$ that takes $\Sigma_{T}$ to $\Sigma_0$ in $T$ units of time. This means that $\Sigma_0\in\mathcal{R}(\Sigma_{T},-D,T)$, which completes the proof.
\end{proof}

This observation yields an additional necessary condition besides the one dictated by Theorem \ref{maintheorem}. We include them both in the statement below.
\begin{corollary}\label{fig:corol}
     If $\Sigma_{T}\in\mathcal{R}(\Sigma_0,D,T)$, then it holds that
    \begin{align*}
       &\log{\blambda_T}\prec \log{\blambda_0^\downarrow}+2T{\blambda_D^\downarrow}, \mbox{ and}\\
       &\log{\blambda_0}\prec \log{\blambda_T^\downarrow}+2T{\blambda_{-D}^\downarrow},
    \end{align*}
    with  $\log{\blambda_{T }}$, $\log{\blambda_0}$, and ${\blambda_D}$ be the vectors of eigenvalues of $\log\Sigma_{T}$, $\log\Sigma_0$, and $D$, respectively. 
\end{corollary}
\begin{proof}
    Suppose $\Sigma_{T}\in\mathcal{R}(\Sigma_0,D,T)$. By Theorem \ref{maintheorem} it holds that 
    $$\log{\blambda_T}\prec \log{\blambda_0^\downarrow}+2T{\blambda_D^\downarrow}.$$
    By Proposition \ref{prop12}, $\Sigma_0\in\mathcal{R}(\Sigma_{T},-D,T)$ so applying Theorem \ref{maintheorem} it must also hold that 
     $$ \log{\blambda_0}\prec \log{\blambda_T^\downarrow}+2T{\blambda_{-D}^\downarrow}.$$ This completes the proof.
\end{proof}

Clearly, the characterization of $\mathcal{R}(cI,T,D)$ in Theorem \ref{thm:2} for the case where $\Sigma_0=cI$ admits a time-symmetric counterpart, that is,  if $\Sigma_{T}=cI$ for $c>0$, then 
    \begin{align*}
    \mathcal{R}(\Sigma_{T},T,-D)=\{\Sigma_0>0~:~\log\blambda_{\Sigma_0}\prec\log\blambda^{\downarrow}_{\Sigma_T}+2T\blambda^{\downarrow}_{-D}\}.
    \end{align*}
    We summarize in Figure \ref{Fig:1} the general relationship between eigenvalue majorization and isospectral controllability in the general case (black) and the special cases when either the initial or terminal covariances are multiples of the identity (red and blue). This captures all the results we have presented so far and serves as a checkpoint for the reader. 

    Note that in the general case where $\Sigma_0$, $\Sigma_{T}$ are arbitrary positive definite covariances, the necessary conditions  outlined in Corollary \ref{fig:corol} need not be sufficient. A trivial example can illustrate this. Suppose $D=0$ and the endpoint covariances are diagonal, disinct, and correspond to permutations of each other. Since $D=0$, the necessary conditions of Corollary \ref{fig:corol} then read 
    \begin{align*}
        \log\blambda_0\prec \log\blambda_{T},~~\log\blambda_{T}\prec \log\blambda_0.
    \end{align*}The above obviously hold since permuted vectors automatically majorize each other (by definition, majorization sorts the vectors before comparing them). Yet the terminal covariance  $\Sigma_{T}$ is clearly outside the reachable set $\mathcal{R}(\Sigma_0,T,0)=\{\Sigma_0\}$. 
    
\begin{figure}\label{Fig:1}
    \centering
    \fbox{\includegraphics[width=.97\linewidth]{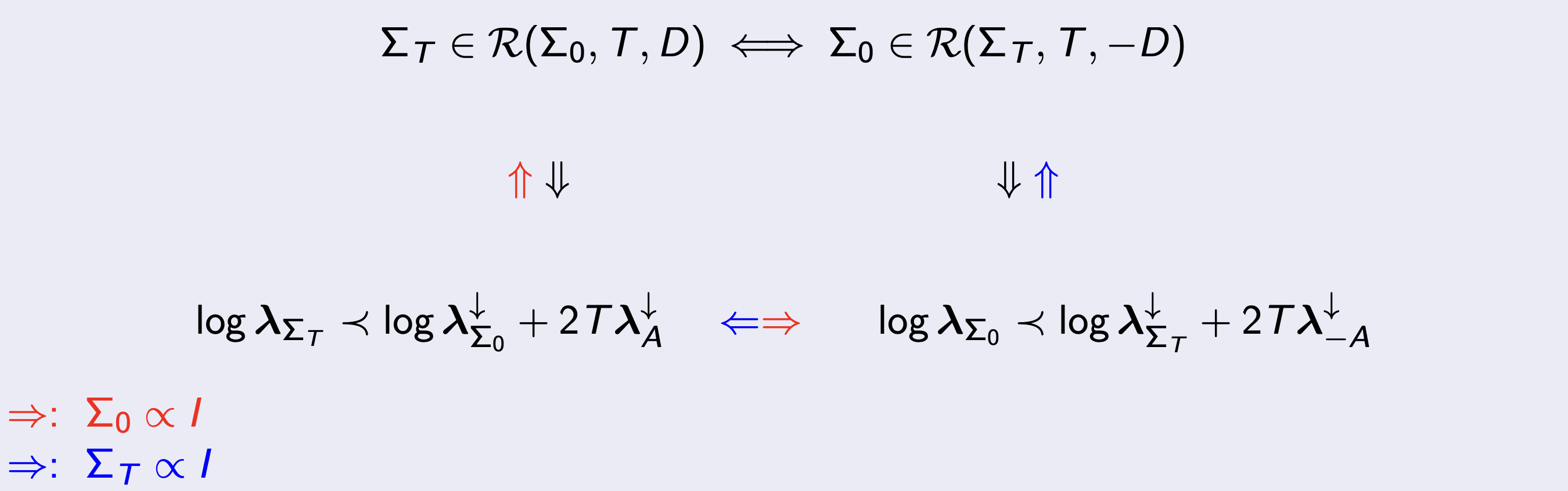}}
    \caption{Relationship between eigenvalue majorization and isospectral controllability in general (black) and in the special cases when either the initial or terminal covariances are multiples of the identity, in red and blue, respectively.}
\end{figure}
We now exploit the time-symmetry properties of the problem to derive a sufficient condition for controllability when the time allows passing through a multiple of the identity.
\begin{theorem}\label{theoremsuff}
    Let $T>0$, $\Sigma_0,\Sigma_{T}>0$. If there exists $c>0$ and $t\in[0,T]$ such that 
    \begin{align*}
    \log(\blambda_T/c)&\prec 2(T-t)\blambda_{D},\\
    \log(\blambda_0/c)&\prec 2t\blambda_{-D},
    \end{align*}
    then $\Sigma_T\in\mathcal{R}(\Sigma_0,T,D)$.
\end{theorem}
\begin{proof}
    From Theorem \ref{thm:2} and the time-symmetry \ref{prop12}, the stated conditions imply that we can reach $cI$ from $\Sigma_0$ in $t$ units of time using $D$, followed by reaching $\Sigma_T$ from $cI$ in $T-t$ units of time using once more $D$, i.e., we have the sequence of steps
    \begin{align*}
        \overbrace{\Sigma_0\xrightarrow[D, t] \, cI \xrightarrow[D,T-t]\,\Sigma_T}^{D,T},
    \end{align*}
    which shows that $\Sigma_T$ is reachable, as claimed.
\end{proof}

Finally, in the theorem below, we note that provided the prescribed spectrum of the traceless $A_t$ is nontrivial (i.e., $A_t$ is not the zero matrix) and assuming a sufficiently long time interval, it is always possible to steer between any two covariances $\Sigma_0,\Sigma_T$ of the same determinant. 

\begin{theorem}[finite-time controllability]
Given $D\neq 0$ diagonal and traceless, $\Sigma_1>0,\Sigma_2>0$ with $\det(\Sigma_1)=\det(\Sigma_2)$, there exists a $T>0$ such that 
 $\Sigma_2 = \mathcal{R}(\Sigma_1,T,D)$.
\end{theorem}
\begin{proof}
    Since $D$ is traceless and nonzero, 
    \begin{align*}
\sum_{i=1}^k\lambda_i({{D}^{\downarrow}})>0,\mbox{ for }k\in\{1,2,\ldots,n-1\}.
    \end{align*}
    (Note that for $k=n$ the sum vanishes since $D$ is traceless.) Thus for any $c>0$ and for a large enough $T$, taking $t=T/2$, the majorization conditions
    \begin{align*}
    \log(\blambda_{\Sigma_2}/c)&\prec 2(T-t)\blambda_{D}=T\blambda_{D},\\
    \log(\blambda_{\Sigma_1}/c)&\prec 2t\blambda_{-D}=T\blambda_{-D},
    \end{align*}
    must hold.
    By Theorem \ref{theoremsuff}, $\Sigma_2\in\mathcal{R}(\Sigma_1,T,D)$, as claimed.
\end{proof}

\begin{remark}
    The conclusion of the above theorem is rather surprising
since finite-time controllability of the {\em bilinear dynamics} in \eqref{eq:lyap} is ascertained from static linear-algebraic conditions, forgoing any geometric control arguments.
\end{remark}

\begin{figure*}[h]
    \centering
    \includegraphics[width=\linewidth]{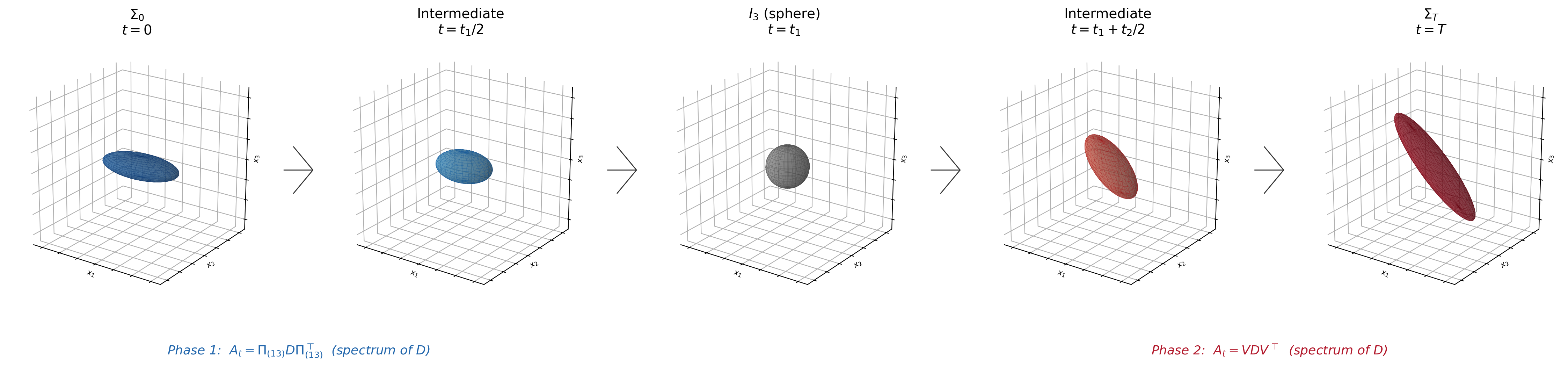}
    \caption{Isospectral steering of a $3\times 3$ covariance from $\Sigma_0$ (elongated along $x_1$) to $\Sigma_T$ (elongated at $45^\circ$ in the $(x_1,x_3)$-plane), passing through the identity $I_3$ (sphere) at $t=t_1$. The control $A_t$ has a fixed spectrum $\{2,0,-2\}$ throughout; only its eigenvectors are modulated. The blue-to-gray transition (Phase~1) contracts the ellipsoid to a sphere, and the gray-to-red transition (Phase~2) re-inflates it into the target shape. Total time $T=1$.}
    \label{fig:ellipsoids}
\end{figure*}

\section{Numerical Example}\label{sec:numerical}
 
We illustrate the constructive nature of the finite-time controllability result with a concrete three-dimensional example. The proof of Theorem~19 is implicitly constructive: it steers $\Sigma_0$ to $\Sigma_T$ by first contracting to the identity (a ``ball'') and then expanding into the target---much as a child might reshape a lump of modeling clay by first rolling it into a sphere and then molding it into a new form.
 
\subsection{Setup}\label{sec:num_setup}
 
Let $n=3$ and fix the traceless diagonal
\[
D=\diag(2,\,0,\,-2),
\]
so that the isospectral control set consists of all symmetric matrices with eigenvalues $\{2,0,-2\}$. We steer between
\begin{align*}
\Sigma_0&=\diag(4,\,1,\,\tfrac14),\\
\Sigma_T&=V\,\diag(9,\,1,\,\tfrac19)\,V^\top,
\end{align*}
where $V$ is a rotation by $\pi/4$ about the $x_2$-axis,
\[
V=\begin{pmatrix}
\tfrac{1}{\sqrt2}&0&\tfrac{1}{\sqrt2}\\[3pt]
0&1&0\\[3pt]
-\tfrac{1}{\sqrt2}&0&\tfrac{1}{\sqrt2}
\end{pmatrix}\in SO(3).
\]
Both covariances have unit determinant. The initial ellipsoid is elongated along $x_1$, while the target is elongated along a direction rotated by $45^\circ$ in the $(x_1,x_3)$-plane.
 
\subsection{Verification of the majorization conditions}
 
We pass through $cI_3$ with $c=1$ (the identity), taking $t_1+t_2=T$. The two majorization conditions (cf.\ Theorem~\ref{theoremsuff}) are:
\begin{align*}
&\text{Phase 1: }\log\blambda(\Sigma_0)\prec 2t_1\blambda_{-D}^\downarrow,\\
&\text{Phase 2: }\log\blambda(\Sigma_T)\prec 2t_2\blambda_D^\downarrow.
\end{align*}
Since $D=\diag(2,0,-2)$, both $\blambda_D^\downarrow=(2,0,-2)$ and $\blambda_{-D}^\downarrow=(2,0,-2)$ (the sorted spectra coincide due to the symmetric placement of the eigenvalues about zero). The eigenvalue vectors are
\begin{align*}
\log\blambda(\Sigma_0)&=(\log 4,\,0,\,-\log 4)\approx(1.386,\,0,\,-1.386),\\
\log\blambda(\Sigma_T)&=(\log 9,\,0,\,-\log 9)\approx(2.197,\,0,\,-2.197).
\end{align*}
The binding majorization constraint (at $k=1$) gives
\begin{align*}
&\text{Phase 1: }\log 4\leq 4t_1\;\Longrightarrow\;t_1\geq \tfrac{\log 4}{4}\approx 0.347,\\
&\text{Phase 2: }\log 9\leq 4t_2\;\Longrightarrow\;t_2\geq \tfrac{\log 9}{4}\approx 0.549.
\end{align*}
Hence the minimum total time is
\[
T_{\min}=\frac{\log 4}{4}+\frac{\log 9}{4}=\frac{\log 36}{4}\approx 0.896.
\]
We select $T=1$, with $t_1=\frac{\log 4}{4}\approx 0.347$ and $t_2=T-t_1\approx 0.653$.
 
\subsection{Constructive control}
 
Following the proof of Theorem~\ref{thm:2}, the control in each phase is a piecewise-constant \emph{permuted} copy of $D$, rotated into the eigenbasis of the target covariance.
 
\medskip
\noindent{\em Phase 1 ($\Sigma_0\to I_3$, duration $t_1$).}
Since $\Sigma_0$ is diagonal with eigenvalues $(4,1,\frac14)$ arranged so that the largest eigenvalue must \emph{shrink} and the smallest must \emph{grow}, the appropriate generator is the permuted copy
\[
A_t^{(1)}=\Pi_{(13)}\,D\,\Pi_{(13)}^\top=\diag(-2,\,0,\,2),\quad t\in[0,t_1],
\]
where $\Pi_{(13)}$ is the permutation matrix that swaps entries $1$ and $3$. This has the same spectrum as $D$ and drives the Lyapunov equation, which amounts to three scalar equations
\[
\frac{d}{dt}\Sigma_{ii}=2(A^{(1)})_{ii}\Sigma_{ii},
\]
for $i\in\{1,2,3\}$,
yielding
\begin{align*}
\Sigma_{11}(t)&=4e^{-4t},\quad
\Sigma_{22}(t)=1,\quad
\Sigma_{33}(t)=\tfrac14 e^{4t}.
\end{align*}
At $t=t_1=\frac{\log 4}{4}$: $\Sigma_{11}=4\cdot\frac14=1$, $\Sigma_{33}=\frac14\cdot 4=1$, confirming $\Sigma_{t_1}=I_3$.
 
\medskip
\noindent{\em Phase 2 ($I_3\to\Sigma_T$, duration $t_2$).}
Starting from $I_3$, we apply
\[
A_t^{(2)}=V\,D\,V^\top,\quad t\in[t_1,T].
\]
In $V$'s eigenbasis, the dynamics is diagonal with rates $(2,0,-2)$, giving eigenvalues
\[
\lambda_i(\Sigma_{t_1+s})=e^{2d_is},\quad s\in[0,t_2],
\]
where $(d_1,d_2,d_3)=(2,0,-2)$. At $s=t_2\approx 0.653$: $e^{4t_2}=e^{\log 9+4(t_2-\log 9/4)}$. Since $t_2>\frac{\log 9}{4}$, the eigenvalues at the terminal time \emph{exceed} $(9,1,\frac19)$. To achieve the exact target, we use a brief corrective sub-phase: apply the swapped generator $V\Pi_{(13)}D\Pi_{(13)}^\top V^\top$ for the remaining time, analogous to Phase~1. The Birkhoff decomposition gives the precise time allocation between the two permuted generators; at the minimum time $t_2=\frac{\log 9}{4}$, only the identity permutation is needed and the control is simply $A_t^{(2)}=VDV^\top$.
 
\subsection{Snapshots of the covariance ellipsoid}
 
Figure~\ref{fig:ellipsoids} shows the unit-level sets $\{\mathbf x:\mathbf x^\top\Sigma_t^{-1}\mathbf x=1\}$ at five instants:
\begin{enumerate}
\item $t=0$: $\Sigma_0=\diag(4,1,\frac14)$, ellipsoid elongated along $x_1$.
\item $t=t_1/2$: $\Sigma=\diag(2,1,\frac12)$, ellipsoid contracted.
\item $t=t_1$: $\Sigma=I_3$, a sphere.
\item $t=t_1+t_2/2$:  $\Sigma=V\diag(3,1,\frac13)V^\top$, elongated in a rotated direction.
\item $t=T$: target $\Sigma_T=V\diag(9,1,\frac19)V^\top$, fully elongated at $45^\circ$.
\end{enumerate}
Throughout both phases, the spectrum of the control $A_t$ remains fixed at $\{2,0,-2\}$; only the eigenvectors change. The determinant remains $\det(\Sigma_t)=1$ at every instant, as guaranteed by the tracelessness of $D$.

\section{Conclusion}\label{sec:conclusion}

We introduced the problem of isospectral steering within the frame of covariance control. 
The problem amounts to the controllability of the differential Lyapunov equation with a time-varying gain that is constrained to have a fixed, prescribed spectrum. 
The constraint on the spectrum arises naturally from optimality conditions when minimizing shear deformation or, more generally, functionals that penalize the spectral spread of the control gain.

Our main contributions are as follows. First, we showed that isospectrality of the control law is not merely a convenient assumption but a consequence of optimality for a broad class of cost functionals of the form $\int_0^T \tr(f(A_t))\,dt$, provided $f'$ is injective on the spectrum of $A_t$ (Section~\ref{sec:shear}). Second, we established that log-majorization of the eigenvalues of the terminal covariance is a necessary condition for reachability under isospectral control in arbitrary dimensions (Theorem~\ref{maintheorem}), leveraging tools from multilinear algebra---specifically, the exterior power functor and compound matrices. Third, we proved that when the initial covariance is a scalar multiple of the identity, this majorization condition is also sufficient (Theorem~\ref{thm:2}), with the proof relying on the Birkhoff--von Neumann theorem for doubly stochastic matrices and the commutativity of permuted diagonal controls. Fourth, by exploiting a time-reversal symmetry of the Lyapunov equation (Proposition~\ref{prop12}), we derived additional necessary conditions (Corollary~\ref{fig:corol}) as well as sufficient conditions for reachability that allow passage through an isotropic intermediate state (Theorem~\ref{theoremsuff}). Finally, we showed that for any nontrivial traceless spectrum, finite-time controllability between arbitrary equi-determinant covariances is always achievable given a sufficiently long time horizon; the argument we provided is constructive and relies on the Birkhoff decomposition of bi-stochastic matrices.

Several directions remain open. A complete characterization of the reachable set $\mathcal{R}(\Sigma_0,D,T)$ for general (non-isotropic) initial covariances $\Sigma_0$ is not yet available; the gap between the necessary majorization conditions and the sufficient conditions presented here warrants further investigation. In addition, the inverse problem proposed in Remark \ref{rem:kalman}---determining which spectral signatures $D$ arise as optimal for a given pair $(\Sigma_0,\Sigma_T)$ and a suitable functional $f$---connects to classical questions in the spirit of Kalman and remains unresolved. Finally, the geometric structure of the reachable set for the full state-transition dynamics $\dot{X}_t = A_t X_t$, as opposed to the Lyapunov equation, involves a richer sub-Riemannian geometry \cite{abdelgalil2025holonomy} on $SL(n,\mathbb{R})$ whose explicit characterization in general is also an open problem.

The approach developed here suggests a useful lens for tackling the non-convex optimization problems that originally motivated our study \cite{sabbagh2025minimizing,Sabbagh26}. Rather than directly solving for the optimal control $\{A_t\}$ that minimizes shear deformations---a problem that is non-convex and may admit multiple solutions---one may instead first characterize the reachable set under the isospectrality constraint that optimality imposes, and then optimize over this geometrically structured set. In effect, the non-convexity of the original problem is absorbed into the geometry of the reachable set, which, as we have shown, is governed by the convex notion of majorization. This ``inversion'' of perspective---deriving the constraint set before optimizing---may provide a useful insight into the structure of the problem.

\bibliographystyle{unsrt}
\bibliography{refs}

\end{document}